\documentclass[oneside,english]{amsart}
\usepackage[OT1]{fontenc}
\usepackage[latin9]{inputenc}
\usepackage{amsthm}
\usepackage{amssymb}

\makeatletter
\numberwithin{equation}{section}
\numberwithin{figure}{section}
\theoremstyle{plain}
\newtheorem{thm}{\protect\theoremname}
  \theoremstyle{definition}
  \newtheorem{defn}[thm]{\protect\definitionname}
  \theoremstyle{plain}
  \newtheorem{prop}[thm]{\protect\propositionname}
  \theoremstyle{plain}
  \newtheorem{lem}[thm]{\protect\lemmaname}
  \theoremstyle{plain}
  \newtheorem{cor}[thm]{\protect\corollaryname}

\makeatother

\usepackage{babel}
  \providecommand{\corollaryname}{Corollary}
  \providecommand{\definitionname}{Definition}
  \providecommand{\lemmaname}{Lemma}
  \providecommand{\propositionname}{Proposition}
\providecommand{\theoremname}{Theorem}

\begin{document}

\title[Functional equation of normalized Shintani L-function]{On the functional equation of the normalized Shintani l-function
of several variables}
\begin{abstract}
In this paper, we introduce the normalized Shintani L-function of
several variables by an integral representation and prove its functional
equation. The Shintani L-function is a generalization to several variables
of the Hurwitz-Lerch zeta function and the functional equation given
in this paper is a generalization of the functional equation of Hurwitz-Lerch
zeta function. In addition to the functional equation, we give special
values of the normalized Shintani L-function at non-positive integers
and some positive integers.
\end{abstract}

\author{Minoru Hirose, Nobuo Sato}

\maketitle

\section{Introduction}

The Shintani zeta functions and the Shintani L-functions are originally
introduced by Takuro Shintani in his study of Hecke L-functions of
totally real algebraic number fields \cite{key-Shintani}. He gave,
for example, special values of Hecke L-functions at non-positive integers,
by representing a Hecke L\emph{-}function as a linear combination
of the Shintani zeta functions. Shintani's original zeta and L-functions
are of one variable, but in this paper we consider Shintani L-functions
of several variables (see \cite{key-Hida}), since sometimes they
have more refined information. 

Functional equations for zeta and L-functions of one variable are
classically known, for example, for Hecke L-functions \cite{key-Hecke}
which are further developed in the Iwasawa-Tate theory \cite{key-Tate}.
Functional equations for zeta and L-functions of several variables
are known, for example, for those associated to prehomogeneous vector
spaces \cite{key-SatoFumihiro}. However, such functional equations
for the Shintani zeta functions or the Shintani L-functions have not
been known in general except for the case of the Hurwitz-Lerch zeta
function \cite{key-Lerch} and the case of the double zeta function\cite{key-Matsumoto}.
In this paper, we show that the normalized Shintani L-functions generally
satisfy nice functional equations, which is a generalization of the
functional equation of the Hurwitz-Lerch zeta function. The functional
equation of the normalized Shintani L-function seems applicable to
refine the functional equations of Hecke L-functions of totally real
number fields.

First of all, we give a definition (Definition \ref{def:Series-expression})
of the Shintani L-function by a multiple Dirichlet series. Then, we
give an integral representation of the Shintani L-function (Proposition
\ref{prop:Integral-representation}) by which we extend the definition
of the Shintani L-function. We then use a similar integral representation
to define the normalized Shintani L-function (Definition \ref{def:Normal Shintani}).
This extension of the definition is necessary to formulate the functional
equation. Throughout this paper, we use the notations 
\begin{align*}
e\left(z\right) & =e^{2\pi iz},\\
\Gamma_{\mathbb{R}}\left(s\right) & =\pi^{-\frac{s}{2}}\Gamma(\frac{s}{2}),\\
\Gamma_{\mathbb{C}}\left(s\right) & =2\left(2\pi\right)^{-s}\Gamma(s),
\end{align*}
and for $u=\left(u_{1},\ldots,u_{r}\right)$, $v=\left(v_{1},\ldots,v_{r}\right)\in\mathbb{R}^{r}$
and $s=\left(s_{1},\ldots,s_{r}\right)\in\mathbb{C}^{r}$, 
\begin{align*}
uv & =u_{1}v_{1}+\cdots+u_{r}v_{r},\\
u^{s} & =u_{1}^{s_{1}}\cdots u_{r}^{s_{r}},\\
1-u & =\left(1-u_{1},\ldots,1-u_{r}\right),
\end{align*}
for simplicity. We also use the notation $\left(0,1\right)$ for the
unit open interval.
\begin{defn}
\label{def:Series-expression}

For $s=(s_{\nu})_{\nu=1}^{r}\in\mathbb{\mathbb{C}}^{r}$ with $\Re(s_{1}+\cdots+s_{r})>r$,
$A=(a_{\nu\mu})_{\nu,\mu=1}^{r}\in M_{r}(\mathbb{R}_{>0})$, and $x=(x_{\nu})_{\nu=1}^{r},y=(y_{\mu})_{\mu=1}^{r}\in(0,1)^{r}$
define the Shintani L-function $L\left(s,A,x,y\right)$ of order $r$
by
\begin{align*}
L(s,A,x,y) & =\sum_{0\leq m_{1},\cdots,m_{r}}\frac{e(m_{1}y_{1}+\cdots+m_{r}y_{r})}{\prod_{\nu=1}^{r}(\sum_{\mu=1}^{r}a_{\nu\mu}(x_{\mu}+m_{\mu}))^{s_{\nu}}}\\
 & =\sum_{0\leq m}\frac{e\left(my\right)}{(A(x+m))^{s}},
\end{align*}
It is easy to check that the series on the right-hand side converges
absolutely in the half space $\Re(s_{1}+\cdots+s_{r})>r$ and conditionally
for $\Re(s_{1}+\cdots+s_{r})>0$.
\end{defn}
$L(s,A,x,y)$ has an integral representation which extends the domain
of definition of $A$ to $GL_{r}(\mathbb{R})$. This expression is
very important in our context, since we cannot avoid $A$ with negative
entries for our purpose. Put, for simplicity, 
\begin{align*}
\varphi(t,x,y) & =\frac{e(itx)}{1-e(y+it)}\\
F\left(t,x,y\right) & =\prod_{\nu=1}^{r}\varphi(t_{\nu},x_{\nu},y_{\nu}).
\end{align*}
Then we have the following integral representation for $L(s,A,x,y)$. 
\begin{prop}
\label{prop:Integral-representation}

For $\Re(s_{1}),\ldots,\Re(s_{r})>0$, the Shintani L-function $L(s,A,x,y)$
of \emph{degree} $r$ has the following integral representation.
\[
L(s,A,x,y)=\frac{2^{r}}{\Gamma_{\mathbb{C}}(s_{1})\cdots\Gamma_{\mathbb{C}}(s_{r})}\int_{0}^{\infty}\cdots\int_{0}^{\infty}F(tA,x,y)t^{s}\frac{dt_{1}}{t_{1}}\cdots\frac{dt_{r}}{t_{r}}
\]

\end{prop}
The proof is easy. Note again that, since the integral on the left
hand side of Proposition \ref{prop:Integral-representation} is defined
for $A\in GL_{r}(\mathbb{R})$, we extend the definition of $L(s,A,x,y)$
by this integral representation. Next we introduce the \emph{normalized}
Shintani L-functions. For a map $\chi:\{1,\cdots,r\}\rightarrow\{0,1\}$,
set the quasi-character $\left|\cdot\right|_{\chi}^{s}$ on $(\mathbb{R}^{r})^{\times}$
by 
\[
\left|t\right|_{\chi}^{s}=\prod_{\nu=1}^{r}{\rm sgn}(t_{\nu})^{\chi(\nu)}\left|t_{\nu}\right|^{s_{\nu}},
\]
where sgn denote the signum function on $\mathbb{R}^{\times}.$ Then
the normalized Shintani L-function attached to $\chi$ is defined
as follows.
\begin{defn}
\label{def:Normal Shintani}For a map $\chi:\{1,\cdots,r\}\rightarrow\{0,1\}$,
define the \emph{normalized} Shintani L\emph{-}function of \emph{parity
type} $\chi$ of degree $r$ by 
\[
L_{\chi}(s,A,x,y)=\frac{1}{\Gamma_{\mathbb{C}}(s_{1})\cdots\Gamma_{\mathbb{C}}(s_{r})}\int_{(\mathbb{R}^{r})^{\times}}F\left(tA,x,y\right)\left|t\right|_{\chi}^{s}\frac{dt_{1}}{t_{1}}\cdots\frac{dt_{r}}{t_{r}},
\]
 and its \emph{completion} by 
\[
\hat{L}_{\chi}(s,A,x,y)=\left|\det(A)\right|^{\frac{1}{2}}\Gamma_{\chi}(s)L_{\chi}(s,A,x,y)
\]
where we set its gamma factor by $\Gamma_{\chi}(s)=\prod_{\nu=1}^{r}\Gamma_{\mathbb{R}}(s_{\nu}+\chi(\nu))$. 

The notion of \emph{parity type} of the normalized Shintani L-function
corresponds to the \emph{infinity type} of the Hecke L-function of
a totally real field. As easily seen, by dividing the domain of integration,
the normalized Shintani L-function is related to the ordinary Shintani
L-function by 
\[
L_{\chi}(s,A,x,y)=2^{-r}\sum_{\sigma\in\{\pm1\}^{r}}\sigma^{1-\chi}L(s,\sigma A,x,y),
\]
where for $\sigma=(\sigma_{\nu})_{\nu=1}^{r}\in\left\{ \pm1\right\} ^{r}$,
we write $\sigma^{1-\chi}=\prod_{\nu=1}^{r}\sigma_{\nu}^{1-\chi(\nu)}$
and $\sigma A=(\sigma_{\nu}a_{\nu\mu})_{\nu,\mu=1}^{r}$. The inversion
formula is simply 
\[
L\left(s,A,x,y\right)=\sum_{\chi}L_{\chi}\left(s,A,x,y\right),
\]
where $\sum_{\chi}$ means the sum over all the parity types of degree
$r.$ Thus, a normalized Shintani L-function is expressible by a finite
sum of Shintani L-functions and vica versa, but the \emph{normalized}
Shintani L\emph{-}function satisfy much simpler functional equations. 
\end{defn}
Now the main theorem.
\begin{thm}
\label{thm:Functional equation}Let $\hat{L}_{\chi}\left(s,A,x,y\right)$
be a complete normalized Shintani $L$-function of degree $r$ and
parity type $\chi$ defined on $s\in\mathbb{C}^{r}$. Set $i_{\chi}=i^{\sum_{\nu=1}^{r}\chi(\nu)}$.
Then $\hat{L}_{\chi}\left(s,A,x,y\right)$ satisfies the following
functional equation.
\[
\hat{L}_{\chi}(s,A,x,y)=i_{\chi}e\left(-xy\right)\hat{L}_{\chi}(1-s,A^{*},y,1-x).
\]
Here, $A^{*}=(A^{t})^{-1}$, the inverse transpose of $A$. 
\end{thm}

\section{A proof of the functional equation}

In this section, we give a proof of Theorem \ref{thm:Functional equation}.
Our proof make use of the (inverse) Fourier transform of $F(tA,x,y)$
and Tate's local functional equation. 
\begin{lem}
\label{lem:Fourier transform of F}We have 
\[
\int_{\mathbb{R}^{r}}F(tA,x,y)e(tk)dt=\frac{i^{r}}{\left|\det A\right|}e(-yx)F(kA^{*},y,1-x).
\]
\end{lem}
\begin{proof}
Put $u=tA$ and $h=kA^{*}$, then $dt=\left|\det A\right|^{-1}du$
and $tk=uh$. Thus we have 
\begin{align*}
\int_{\mathbb{R}^{r}}F(tA,x,y)e(tk)dt & =\frac{1}{\left|\det A\right|}\int_{\mathbb{R}^{r}}F(u,x,y)e(uh)du\\
 & =\frac{1}{\left|\det A\right|}\prod_{i=1}^{r}\int_{\mathbb{R}}\varphi(u_{i},x_{i},y_{i})e(u_{i}h_{i})du_{i},
\end{align*}
 by which we reduce the proof to one variable case. For $\lambda\in\mathbb{R}_{>0}$,
consider the positively oriented rectangular contour $C_{\lambda}$
consisting of the intervals $I_{0,\lambda}=[-\lambda,\lambda]$, $I_{1,\lambda}=[\lambda,\lambda+i]$,
$I_{2,\lambda}=[-\lambda+i,\lambda+i]$ and $I_{3,\lambda}=[-\lambda,-\lambda+i]$.
Then we have 
\[
\int_{I_{2,\lambda}}\varphi(u,x,y)e(uh)du=-e(-x+ih)\int_{I_{0,\lambda}}\varphi(u,x,y)e(uh)du,
\]
so that 
\begin{align*}
\left(1-e(-x+ih)\right)\int_{I_{0,\lambda}}+\int_{I_{1,\lambda}}+\int_{I_{3,\lambda}} & =\int_{C_{\lambda}}\\
 & =ie(-y(x-ih)).
\end{align*}
Thus, by letting $\lambda\rightarrow+\infty$ here, we have 
\begin{align*}
\int_{\mathbb{R}}\varphi(u,x,y)e(uh)du & =\frac{ie(-y(x-ih))}{1-e(-(x-ih))}\\
 & =ie(-yx)\varphi(h,y,1-x)
\end{align*}
from which we obtain Lemma \ref{lem:Fourier transform of F}.
\end{proof}
Now, we recall Tate's local functional equation for $\mathbb{R}$.
\begin{lem}
\label{lemma:Tate's local functional equation}Let $\Phi\in\mathcal{S}(\mathbb{R}^{r})$
and $\hat{\Phi}\in\mathcal{S}(\mathbb{R}^{r})$ be the inverse Fourier
transform of $\Phi$. Then, for $0<\Re s_{1},\ldots,\Re s_{r}<1$,
we have 
\[
\frac{\int_{(\mathbb{R}^{r})^{\times}}\Phi(t)\left|t\right|_{\chi}^{s}d^{\times}t}{\Gamma_{\chi}(s)}=i_{\chi}^{-1}\frac{\int_{(\mathbb{R}^{r})^{\times}}\hat{\Phi}(t)\left|t\right|_{\chi}^{1-s}d^{\times}t}{\Gamma_{\chi}(1-s)},
\]
where $d^{\times}t=\prod_{i=1}^{r}\left|t_{i}\right|^{-1}dt_{i}$
is a Haar measure of $(\mathbb{R}^{r})^{\times}$.
\end{lem}
Since $\mathcal{S}(\mathbb{R})^{\otimes r}$ is dense in $\mathcal{S}(\mathbb{R}^{r})$,
Lemma \ref{lemma:Tate's local functional equation} immediately follows
from Tate's local functional equation. Now, by a simple calculation,
we see that 
\begin{align*}
\left|\det(A)\right|^{-\frac{1}{2}}\hat{L}_{\chi}(s,A,x,y) & =\frac{\Gamma_{\chi}(s)}{\prod_{i=1}^{r}\Gamma_{\mathbb{C}}(s_{i})}\int_{(\mathbb{R}^{r})^{\times}}F(tA,x,y)\left|t\right|_{\chi}^{s}\prod_{i=1}^{r}\frac{dt_{i}}{t_{i}}\\
 & =\frac{1}{\Gamma_{1-\chi}(s)}\int_{(\mathbb{R}^{r})^{\times}}F(tA,x,y)\left|t\right|_{1-\chi}^{s}d^{\times}t
\end{align*}
As it is clear that $F(tA,x,y)\in\mathcal{S}(\mathbb{R}^{r})$ for
$A\in GL_{r}(\mathbb{R})$ and $x,y\in(0,1)^{r}$, we have by Lemma
\ref{lem:Fourier transform of F} and Lemma \ref{lemma:Tate's local functional equation},
\[
\frac{\int_{(\mathbb{R}^{r})^{\times}}F(tA,x,y)\left|t\right|_{\chi}^{s}d^{\times}t}{\Gamma_{\chi}(s)}=\frac{i_{1-\chi}}{\left|\det A\right|}e(-xy)\frac{\int_{(\mathbb{R}^{r})^{\times}}F(tA^{*},y,1-x)\left|t\right|_{\chi}^{1-s}d^{\times}t}{\Gamma_{\chi}(1-s)}.
\]
Replacing $\chi$ by $1-\chi$ in the equality above, we obtain the
functional equation for $\hat{L}_{\chi}(s,A,x,y)$ for $0<\Re s_{1},\ldots,\Re s_{r}<1$.
Since $\hat{L}_{\chi}(s,A,x,y)$ is a holomorphic function on $\mathbb{C}^{r}$
(see Corollary \ref{cor:holom}), we can conclude that the functional
equation holds for $\mathbb{C}^{r}$.

\section{Several properties of the normalized Shintani L-function}

In this section, we note several additional properties of the normalized
Shintani L-functions. 
\begin{prop}
\emph{\label{prop:Dirichlet series for r=00003D1 case}In the case
of degree $r=1$, $L_{\chi}(s,A,x,y)$ admits the following Dirichlet
series expressions for $\Re s>1$.}
\[
L_{\chi}(s,A,x,y)=\frac{{\rm sgn}(A)}{2}\sum_{m\in\mathbb{Z}}\frac{e(my)}{\left|A(x+m)\right|_{\chi}^{s}}.
\]
\end{prop}
\begin{proof}
We have 
\[
L_{\chi}(s,A,x,y)=\frac{1}{2}L(s,A,x,y)-\frac{(-1)^{\chi(1)}}{2}L(s,-A,x,y).
\]
 Since $\varphi(-tA,x,y)=-e(-y)\varphi(tA,1-x,1-y)$, we have 
\begin{align*}
L(s,-A,x,y) & =\frac{-e(-y)}{\Gamma_{\mathbb{C}}(s)}\int_{0}^{\infty}\varphi(tA,1-x,1-y)dt\\
 & =-e(-y)L(s,A,1-x,1-y).
\end{align*}
Hence, we obtain 
\[
L_{\chi}(s,A,x,y)=\frac{{\rm sgn}(A)^{1-\chi\left(1\right)}}{2}\left(L(s,\left|A\right|,x,y)+(-1)^{\chi(1)}e(-y)L(s,|A|,1-x,1-y)\right)
\]
 from which Proposition \ref{prop:Dirichlet series for r=00003D1 case}
follows. 
\end{proof}
From Proposition \ref{prop:Dirichlet series for r=00003D1 case},
we see that when $r=1$, $L_{\chi}(s,A,x,y)$ admits a Dirichlet series
expression, which is expressible by a combination of Hurwitz-Lerch
zeta functions. However, for $r\geq2,$ we have no Dirichlet series
expressions for the \emph{normalized} Shintani L-functions. This is
because we have no Dirichlet series expressions for $L(s,A,x,y)$
when the matrix $A$ has both positive and negative entries. Note
that though the integral representation in Proposition \ref{prop:Integral-representation}
is only valid for $x\in(0,1)$, the Dirichlet series expression in
Proposition \ref{prop:Dirichlet series for r=00003D1 case} is valid
for $x\in\mathbb{R}\setminus\mathbb{Z}$ satisfying the property
\[
L_{\chi}(s,A,x+k,y+l)=e(-ky)L_{\chi}(s,A,x,y)
\]
for $k,l\in\mathbb{Z}$. Moreover, the function $\hat{L}_{\chi}(s,A,x,y)=\left|\det(A)\right|^{\frac{1}{2}}\Gamma_{\chi}(s)L_{\chi}(s,A,x,y)$
extended by this formula still satisfy the same functional equation
as in Theorem \ref{thm:Functional equation}, since 
\begin{align*}
\hat{L}_{\chi}\left(s,A,x,y\right) & =e\left(-\left[x\right]y\right)\hat{L}_{\chi}\left(s,A,x-\left[x\right],y-\left[y\right]\right)\\
 & =i_{\chi}e\left(-\left[x\right]y-\left(x-\left[x\right]\right)\left(y-\left[y\right]\right)\right)\hat{L}_{\chi}\left(1-s,A^{*},y-\left[y\right],1-x+\left[x\right]\right)\\
 & =i_{\chi}e\left(-xy\right)\hat{L}_{\chi}\left(1-s,A^{*},y,1-x\right),
\end{align*}
where $\left[x\right]$ denotes the unique integer satisfying $\left[x\right]\leq x<\left[x\right]+1.$
Conversely, we have the following proposition.
\begin{prop}
\label{prop:The-unique-extension}The unique extension of the definition
of $L_{\chi}(s,A,x,y)$ to $x,y\in(\mathbb{R}\setminus\mathbb{Z})^{r}$
that satisfy the functional equation 
\[
\hat{L}_{\chi}(s,A,x,y)=i_{\chi}e(-xy)\hat{L}_{\chi}(1-s,A^{*},y,1-x),
\]
 and the property $L_{\chi}(s,A,x,y+k)=L_{\chi}(s,A,x,y)$, is given
by 
\[
L_{\chi}(s,A,x,y)=e(-[x]y)L_{\chi}\left(s,A,x-[x],y-[y]\right),
\]
\textup{where $[x]=([x_{1}],\ldots,[x_{r}])$.}
\end{prop}
We extend the definition of $L_{\chi}(s,A,x,y)$ quasiperiodically
as in Proposition \ref{prop:The-unique-extension}. Now, complementarily
let us define the functions $G\left(t,x,y\right)$ and $R_{\chi}(s,A,x,y)$
by 
\begin{align*}
G\left(t,x,y\right) & =e\left(xy\right)F\left(t,x,y\right)\\
 & =\prod_{\nu=1}^{r}\frac{e\left(\left(y_{\nu}+it_{\nu}\right)x_{\nu}\right)}{\left(1-e\left(y_{\nu}+it_{\nu}\right)\right)}
\end{align*}
and $ $
\[
R_{\chi}(s,A,x,y)=e(xy)L_{\chi}(s,A,x,y).
\]
 Then the quasiperiodicity and the functional equation become 
\begin{align*}
L_{\chi}(s,A,x+k,y+l) & =e(-ky)L_{\chi}(s,A,x,y),\\
R_{\chi}(s,A,x+k,y+l) & =e(xl)R_{\chi}(s,A,x,y),
\end{align*}
for $k,l\in\mathbb{Z}^{r},$ and 
\begin{align*}
\hat{R}_{\chi}(s,A,x,y) & =i_{\chi}^{\pm1}\hat{L}_{\chi}(1-s,A^{*},\pm y,\mp x),\\
\hat{L}_{\chi}(s,A,x,y) & =i_{\chi}^{\pm1}\hat{R}_{\chi}(1-s,A^{*},\pm y,\mp x),
\end{align*}
 with $\hat{R}_{\chi}(s,A,x,y)=\Gamma_{\chi}(s)R_{\chi}(s,A,x,y)$. 

The following proposition gives a differential property of $L_{\chi}(s,A,x,y)$
and $R_{\chi}(s,A,x,y)$.
\begin{prop}
\label{prop:diff}Let $A=(a_{\mu\nu})_{\mu,\nu=1}^{r},A^{*}=(a_{\mu\nu}^{*})_{\mu,\nu=1}^{r}$
. Then we have 
\begin{align*}
\frac{\partial}{\partial x_{\nu}}L_{\chi}(s,A,x,y) & =-\sum_{\mu=1}^{r}a_{\mu\nu}s_{\mu}L_{\chi+1_{\mu}}(s+1_{\mu},A,x,y),\\
\frac{\partial}{\partial y_{\nu}}R_{\chi}(s,A,x,y) & =2\pi i\sum_{\mu=1}^{r}a_{\mu\nu}^{*}R_{\chi+1_{\mu}}(s-1_{\mu},A,x,y),
\end{align*}
or equivalently, 
\begin{align*}
\left(\sum_{\nu=1}^{r}a_{\mu\nu}^{*}\frac{\partial}{\partial x_{\nu}}\right)L_{\chi}(s,A,x,y) & =-s_{\mu}L_{\chi+1_{\mu}}(s+1_{\mu},A,x,y),\\
\left(\sum_{\nu=1}^{r}a_{\mu\nu}\frac{\partial}{\partial y_{\nu}}\right)R_{\chi}(s,A,x,y) & =2\pi i\, R_{\chi+1_{\mu}}(s-1_{\mu},A,x,y).
\end{align*}
Here, we identify a parity type $\chi$ with the vector $\chi=(\chi(\nu))_{\nu=1}^{r}\in(\mathbb{Z}/2\mathbb{Z})^{r}$
and we denote by $1_{\mu}=(0,\ldots,1,\ldots,0)$ a vector whose entries
are $0$ except for the $\mu$-th entry being $1$. 
\end{prop}
Proposition \ref{prop:diff} follows from the observation 
\begin{align*}
\frac{\partial}{\partial x_{\nu}}F\left(t,x,y\right) & =-2\pi t_{\nu}F\left(t,x,y\right),\\
\frac{\partial}{\partial y_{\nu}}G\left(t,x,y\right) & =-i\frac{\partial}{\partial t_{\nu}}G\left(t,x,y\right),
\end{align*}
and integration by parts. Using these differential operators, we now
see that the functional equation of one parity type can generate the
functional equation of any parity type. 

Finally we give special values of $L_{\chi}(s,A,x,y)$ at non-positive
integers and at some positive integers. Let us first define the multiple
Bernoulli-like number $B_{k}(A,x,y)$.
\begin{defn}
For $k\in\mathbb{Z}_{\geq0}^{r},$ we define $B_{k}(A,x,y)$ by the
coefficients of the Taylor expansion of $F(tA,x,y)$ i.e. 
\[
F(tA,x,y)=\sum_{k\in\mathbb{Z}_{\geq0}^{r}}B_{k}(A,x,y)\frac{(-2\pi t)^{k}}{k!},
\]
where for multi-index $k=(k_{1},\ldots,k_{r}),$ we put $t^{k}=\prod_{\nu=1}^{r}t_{\nu}^{k_{\nu}}$
and $k!=\prod_{\nu=1}^{r}k_{\nu}!$.
\end{defn}
Then we have the following proposition.
\begin{prop}
\textup{\label{prop:specialvals}For $k\in\mathbb{Z}_{\geq0}^{r},$
we have 
\[
L_{\chi}(-k,A,x,y)=\begin{cases}
B_{k}(A,x,y) & \mbox{if }k\equiv1-\chi\pmod{2\mathbb{Z}^{r}}\\
0 & \mbox{otherwise}
\end{cases},
\]
 and for $k\in\mathbb{Z}_{>0}^{r}$ such that $k\equiv\chi\pmod{2\mathbb{Z}^{r}}$,
\[
L_{\chi}\left(k,A,x,y\right)=\left|\det(A)\right|^{-1}e\left(-xy\right)B_{k-1}\left(A^{*},y,1-x\right)\frac{\left(2\pi i\right)^{k}}{2^{r}\left(k-1\right)!}.
\]
 }\end{prop}
\begin{proof}
As one can easily check, the Shintani L\emph{-}function is expressed
in terms of the contour integral 
\[
\prod_{\nu=1}^{r}\left(\frac{2}{\Gamma_{\mathbb{C}}\left(s_{\nu}\right)\left(e\left(s_{\nu}\right)-1\right)}\right)\int_{L}\cdots\int_{L}F\left(tA,x,y\right)t^{s}\frac{dt_{1}}{t_{1}}\cdots\frac{dt_{r}}{t_{r}},
\]
where $L$ is a contour which starts and ends at $+\infty$ and circles
the origin once counterclockwise without encircling any poles other
than the origin. By the formula $\Gamma_{\mathbb{C}}(s)\Gamma_{\mathbb{C}}(1-s)=4i(e(s/2)-e(-s/2))^{-1},$
we see that 
\begin{align*}
L(s,A,x,y) & =\prod_{\nu=1}^{r}\left(\frac{e\left(-\frac{s_{\nu}}{2}\right)\Gamma_{\mathbb{C}}(1-s_{\nu})}{2i}\right)\int_{L}\cdots\int_{L}F\left(tA,x,y\right)t^{s}\frac{dt_{1}}{t_{1}}\cdots\frac{dt_{r}}{t_{r}}.
\end{align*}
 From the definition of $B_{k}\left(A,x,y\right)$ and the residue
theorem, we find that, for $k\in\mathbb{Z}_{\geq0}^{r}$, 
\begin{align*}
L(-k,A,x,y) & =(2\pi i)^{r}\prod_{\nu=1}^{r}\left(\frac{e\left(\frac{k_{\nu}}{2}\right)\Gamma_{\mathbb{C}}(1+k_{\nu})}{2i}\right)B_{k}(A,x,y)\frac{(-2\pi)^{k}}{k!}\\
 & =B_{k}(A,x,y).
\end{align*}
Using the relation $L_{\chi}\left(s,A,x,y\right)=2^{-r}\sum_{\sigma\in\left\{ \pm1\right\} ^{r}}\sigma^{1-\chi}L\left(s,\sigma A,x,y\right),$
we have 
\begin{align*}
L_{\chi}\left(-k,A,x,y\right) & =2^{-r}\sum_{\sigma\in\left\{ \pm1\right\} ^{r}}\sigma^{1-\chi}B_{k}(\sigma A,x,y)\\
 & =2^{-r}B_{k}(A,x,y)\sum_{\sigma\in\left\{ \pm1\right\} ^{r}}\sigma^{1-\chi+k}.
\end{align*}
Hence we have proved the former statement of the Proposition \ref{prop:specialvals}.
The latter statement is proved by applying the functional equation
for $L_{\chi}(s,A,x,y)$ to the former statement. 
\end{proof}
From the contour integral representation of $L_{\chi}(s,A,x,y)$ in
the proof of Proposition \ref{prop:specialvals}, the following fact
follows immediately.
\begin{cor}
\label{cor:holom}$L_{\chi}(s,A,x,y)$ is a holomorphic function on
$s\in\mathbb{C}^{r}$.\end{cor}

\end{document}